\documentclass{siamart1116}
\usepackage{amsfonts}
\usepackage{accents}
\usepackage[margin=1in, letterpaper]{geometry}
\usepackage{graphicx}
\usepackage{bmpsize}
\usepackage{subcaption}
\usepackage{float}
\usepackage{epsfig}
\usepackage{multicol}
\usepackage{cite}
\usepackage{adjustbox,lipsum}

\newcommand{\vertiii}[1]{{\left\vert\kern-0.25ex\left\vert\kern-0.25ex\left\vert #1\right\vert\kern-0.25ex\right\vert\kern-0.25ex\right\vert}}
\begin{document}

\title{A Discrete Hopf Interpolant and Stability of the Finite Element Method for Natural Convection}
\author{J. A. Fiordilino\footnote{University of Pittsburgh, Department of Mathematics, Pittsburgh, PA 15260} and A. Pakzad\footnotemark[1]}
\date{Updated: 7/11/17}
\maketitle





\begin{abstract}
The temperature in natural convection problems is, under mild data assumptions, uniformly bounded in time. This property has not yet been proven for the standard finite element method (FEM) approximation of natural convection problems with nonhomogeneous partitioned Dirichlet boundary conditions, e.g., the differentially heated vertical wall and Rayleigh-B\'{e}nard problems.  For these problems, only stability in time, allowing for possible exponential growth of $\| T^{n}_{h} \| $, has been proven using Gronwall's inequality.  Herein, we prove that the temperature approximation can grow at most linearly in time provided that the first mesh line in the finite element mesh is within $\mathcal{O} (Ra^{-1})$ of the nonhomogeneous Dirichlet boundary.
\end{abstract}


\section{Introduction}
Natural convection of a fluid driven by heating a side wall or the bottom wall is a classic problem in fluid mechanics that is still of technological and scientific importance.  The temperature in this problem is uniformly bounded in time ($\| T(t) \| \leq C < \infty$) under mild data assumptions.  However, when this often analyzed problem is approximated by standard FEM, all available stability bounds, e.g. \cite{Si,Si2,Wu}, for the temperature exhibit exponential growth in time unless the heat transfer through the solid container is included in the model, e.g. \cite{Boland}.  Moreover, even in the stationary case, stability estimates can yield extremely restrictive mesh conditions ($h= \mathcal{O}(Ra^{-30/(6-d)}))$, e.g. \cite{Colm}.

In this paper, we prove that, without the aforementioned restrictions, the temperature approximation is bounded sub-linearly in terms of the simulation time $t^{\ast}$ provided that the first mesh line in the finite element mesh is within $\mathcal{O} (Ra^{-1})$ of the heated wall; that is, $\|T^{n}_{h}\| \leq C\sqrt{t^{\ast}}$.  In practice, numerical simulations are carried out on a graded mesh \cite{Christon,Ince,Manzari,Massarotti} due to the interaction between the boundary layer, which is $\mathcal{O} (Ra^{-1/4})$ in the laminar regime \cite{Gill}, and the core flow.  In particular, several mesh points are placed within the boundary layer, which encompasses the internal core flow.  Although our condition is more restrictive, this may be due to a gap in the analysis and, none-the-less, it is indicative of the value of graded meshes for stability as well as accuracy.

Consider natural convection within an enclosed cavity.  Let $\Omega \subset \mathbb{R}^d$ (d=2,3) be a convex polyhedral domain with boundary $\partial \Omega$.  The boundary is partitioned such that $\partial \Omega = \overline{\Gamma_{1}}  \cup \overline{\Gamma_{2}}$ with $\Gamma_{1} \cap \Gamma_{2} =\emptyset$ and $|\Gamma_H \cup \Gamma_N | = |\Gamma_{1}| > 0$.  Given $u(x,0) = u^{0}(x)$ and $T(x,0) = T^{0}(x)$, let $u(x,t):\Omega \times (0,t^{\ast}] \rightarrow \mathbb{R}^{d}$, $p(x,t):\Omega \times (0,t^{\ast}] \rightarrow \mathbb{R}$, and $T(x,t):\Omega \times (0,t^{\ast}] \rightarrow \mathbb{R}$ satisfy
\begin{align}
	u_{t} + u \cdot \nabla u -Pr \Delta u + \nabla p &= PrRa\xi T + f \; \; in \; \Omega, \label{ndeq1u} \\
	\nabla \cdot u &= 0 \; \; in \; \Omega,  \\
	T_{t} + u \cdot \nabla T - \Delta T &= \gamma \; \; in \; \Omega, \label{ndeq1T} \\
	u = 0 \; \; on \; \partial \Omega,  \; \; \;
	T = 1 \; \; on \; \Gamma_{N}, \; \; T = 0 \; \; on \; \Gamma_{H}, \; \; \;
	n \cdot \nabla T &= 0 \; \; on \; \Gamma_{2}. \label{ndeq1bc}
\end{align}
\noindent Here $n$ denotes the usual outward normal, $\xi = g/|g|$ denotes the unit vector in the direction of gravity, $Pr$ is the Prandtl number, and $Ra$ is the Rayleigh number.  Further, $f$ and $\gamma$ are the body force and heat source, respectively.

In Sections 2 and 3, we collect necessary mathematical tools and present common numerical schemes.  In Section 4, the major results are proven.  In particular, it is shown that provided the first mesh line in the finite element mesh is within $\mathcal{O} (Ra^{-1})$ of the heated wall, then the computed velocity, pressure, and temperature are stable allowing for sub-linear growth in $t^{\ast}$ (Theorems \ref{t1} and \ref{t2}).  Conclusions are presented in Section 5.
\section{Mathematical Preliminaries}
The $L^{2} (\Omega)$ inner product is $(\cdot , \cdot)$ and the induced norm is $\| \cdot \|$.  Moreover, for any subset $\Omega \neq O \subset \mathbb{R}^{d}$ we define the $L^{2}$ inner product $(\cdot , \cdot)_{L2(O)}$ and norm $\| \cdot \|_{L2(O)}$.  Define the Hilbert spaces,
\begin{align*}
	X &:= H^{1}_{0}(\Omega)^{d} = \{ v \in H^{1}(\Omega)^d : v = 0 \; on \; \partial \Omega \}, \;
	Q := L^{2}_{0}(\Omega) = \{ q \in L^{2}(\Omega) : \int_{\Omega} q dx = 0 \}, \\
	W_{\Gamma_{1}} &:= \{ S \in H^{1}(\Omega) : S = 0 \; on \; \Gamma_{1} \}, \; W := H^{1}(\Omega), \;	V := \{ v \in X : (q,\nabla \cdot v) = 0 \; \forall q \in Q \}.
\end{align*}
The explicitly skew-symmetric trilinear forms are denoted:
\begin{align*}
	b(u,v,w) &= \frac{1}{2} (u \cdot \nabla v, w) - \frac{1}{2} (u \cdot \nabla w, v) \; \; \; \forall u,v,w \in X, \\
	b^{\ast}(u,T,S) &= \frac{1}{2} (u \cdot \nabla T, S) - \frac{1}{2} (u \cdot \nabla S, T) \; \; \; \forall u \in X, \; T,S \in W.
\end{align*}
\noindent They enjoy the following useful properties.
\begin{lemma} \label{l1}
There are constants $C_{1}$ and $C_{2}$ such that for all u,v,w $\in$ X and T,S $\in W$, $b(u,v,w)$ and $b^{\ast}(u,T,S)$ satisfy
\begin{align*}
b(u,v,w) &= (u \cdot \nabla v, w) + \frac{1}{2} ((\nabla \cdot u)v, w), \\
b^{\ast}(u,T,S) &= (u \cdot \nabla T, S) + \frac{1}{2} ((\nabla \cdot u)T, S), \\
b(u,v,w) &\leq C_{1} \| \nabla u \| \| \nabla v \| \| \nabla w \|, \\
b^{\ast}(u,T,S) &\leq C_{2} \| \nabla u \| \| \nabla T \| \| \nabla S \|.
\end{align*}
\begin{proof}
See Lemma 18 p. 123 of \cite{Layton}.
\end{proof}
\end{lemma}
\subsection{Finite Element Preliminaries}
Consider a regular mesh $\Omega_{h} = \{K\}$ of $\Omega$ with maximum triangle diameter length $h$.  Let $X_{h} \subset X$, $Q_{h} \subset Q$, $W_{h} \subset W$, and $W_{\Gamma_{1},h} \subset W_{\Gamma_{1}}$ be conforming finite element spaces consisting of continuous piecewise polynomials of degrees \textit{j}, \textit{l}, \textit{j}, and \textit{j}, respectively.  Furthermore, we consider those spaces for which the discrete inf-sup condition is satisfied,
\begin{equation} \label{infsup} 
\inf_{q_{h} \in Q_{h}} \sup_{v_{h} \in X_{h}} \frac{(q_{h}, \nabla \cdot v_{h})}{\| q_{h} \| \| \nabla v_{h} \|} \geq \beta > 0,
\end{equation}
\noindent where $\beta$ is independent of $h$.  The space of discretely divergence free functions is defined by 
\begin{align*}
	V_{h} := \{v_{h} \in X_{h} : (q_{h}, \nabla \cdot v_{h}) = 0, \forall q_{h} \in Q_{h}\}
\end{align*}
and accompanying dual norm
$$	\| w \|_{V_{h}^{\ast}} = \sup_{v_{h} \in V_{h}}\frac{(w,v_{h})}{\|\nabla v_{h}\|}.$$
The continuous time, finite element in space weak formulation of the system (\ref{ndeq1u}) - (\ref{ndeq1bc}) is: Find $u_{h}:[0,t^{\ast}] \rightarrow X_{h}$, $p_{h}:[0,t^{\ast}] \rightarrow Q_{h}$, $T_{h}:[0,t^{\ast}] \rightarrow W_{h}$ for a.e. $t \in (0,t^{\ast}]$ satisfying:
\begin{align} \label{weak:ueq}
(u_{h,t},v_{h}) + b(u_{h},u_{h},v_{h}) + Pr(\nabla u_{h},\nabla v_{h}) - (p_{h}, \nabla \cdot v_{h}) &= PrRa(\gamma T_{h},v_{h}) + (f,v_{h}) \; \; \forall v_{h} \in X_{h}, \\
(q_{h}, \nabla \cdot u_{h}) &= 0 \; \; \forall q_{h} \in Q_{h}, \\
(T_{h,t},S_{h}) + b^{\ast}(u_{h},T_{h},S_{h}) + (\nabla T_{h},\nabla S_{h}) &= (\gamma,S_{h}) \; \; \forall S_{h} \in W_{h,\Gamma_{1}}. \label{weak:Teq}
\end{align}

\subsection{Construction of the discrete Hopf extension}
The mesh condition $h= \mathcal{O}(Ra^{-30/(6-d)})$ from \cite{Colm} arises from the use of the Scott-Zhang interpolant of degree $j$.  To improve upon this condition, we develop a special interpolant for the upcoming analysis.  We construct it as follows:
\\ \indent \textbf{Step one:} Consider those mesh elements $K$ such that $K \cap \Gamma_{1} \neq \emptyset$.  Enumerate these mesh elements from 1 to $l'$. 
\vspace{3pt} 
\\ \indent \textbf{Step two:} $\forall \,1\leq l \leq l'$, let	$\{\phi^{l}_{k}\}^{d+1}_{k=1}$ be the usual piecewise linear hat functions with $supp \; \phi^{l}_{k} \subset K_{l}$ . 
\vspace{3pt}  
\\ \indent \textbf{Step three:} Fix $l$, select those $\phi^{l}_{k}$ such that $\phi^{l}_{k}(x) = 1$ for $x \in K_{l} \cap \Gamma_{1}$.
\vspace{3pt} 
\\ \indent \textbf{Step four:} Define $\psi_{i}$ such that $\{\psi^{i}\}^{i'}_{i=1} = \{\phi^{l}_{k}\}^{k',l'}_{k,l=1}$.
\vspace{3pt} 
\\ \indent \textbf{Step five:} Define $\tau = \sum^{i'}_{i=1} \tilde{T}^{i}\psi^{i}$ where $-\infty < \tilde{T}_{min} \leq \tilde{T}^{i} \leq \tilde{T}_{max} < \infty$ are arbitrary constants.  
\\ \noindent Then,
\begin{theorem} \label{l2}
Suppose $\tilde{T}:\Gamma_{1} \rightarrow \mathbb{R}$ is a piecewise linear function defined on $\Gamma_{1}$.  The discrete Hopf extension $\tau:\Omega \rightarrow \mathbb{R}$ satisfies
\begin{align*}
\tau (x) = \tilde{T} \; on \; \Gamma_{1},
\\ \tau(x) = 0 \; on \; \Omega - \cup^{l'}_{l=1} K_{l}.
\end{align*}
 Moreover, let $\delta = \max_{1\leq l \leq l'} h_{l}$.  Then, the following estimate holds: $\forall \epsilon > 0$, $\forall (\chi_{1},\chi_{2}) \in (X_{h},W_{h})$
\begin{align}
\vert b^{\ast}(\chi_{1},\tau,\chi_{2}) \vert \leq C \delta \Big( \epsilon^{-1} \| \nabla \chi_{1} \|^{2} + \epsilon \| \nabla \chi_{2} \|^{2} \Big). \label{tau3}
\end{align}
\end{theorem}
\begin{proof}
The properties are a consequence of the construction.  For the estimate (\ref{tau3}), it suffices to consider $\vert b^{\ast}(\chi_{1},\tilde{T}^{i}\psi^{i},\chi_{2}) \vert$ where $\tilde{T}^{i} = \tilde{T}(x_{i})$ is the corresponding nodal value of $\tilde{T}$.  For each $\psi^{i}$ there is a corresponding mesh element $K_{l}$ such that $supp \; \psi^{i} \subset K_{l}$.  Let $\hat{K} \subset \mathbb{R}^{d}$ be the reference element and $F_{K_{l}}:\hat{K} \rightarrow K_{l}$ the associated affine transformation given by $x = F_{K_{l}}\hat{x} = B_{K_{l}}\hat{x} + b_{K_{l}}$.  We will utilize the operator norm $\| \cdot \|_{op}$ and the Euclidean norm $\vert \cdot \vert_{2}$ below. 
\\ \indent Consider $\frac{1}{2}\vert (\chi_{1} \cdot \nabla \tilde{T}^{i}\psi^{i},\chi_{2}) \vert$, the estimate for $\frac{1}{2}\vert (\chi_{1} \cdot \nabla \chi_{2}, \tilde{T}^{i}\psi^{i}) \vert$ follows analogously.  Transform to the reference element, use standard FEM estimates, the Cauchy-Schwarz inequality, and equivalence of norms.  Then,
\begin{align} \label{mainestimate}
\frac{1}{2}\vert (\chi_{1}\cdot \nabla \tilde{T}^{i}\psi^{i},\chi_{2}) \vert &= \frac{|\tilde{T}^{i}||det(B_{K_{l}})|}{2} \vert\int_{\hat{K}} \hat{\chi_{1}}\cdot B^{-T}_{K_{l}}\hat{\nabla} \hat{\psi^{i}} \hat{\chi_{2}} d\hat{x} \vert 
\\ &\leq \frac{|\tilde{T}^{i}||det(B_{K_{l}})|}{2} \| B^{-T}_{K_{l}}\|_{op} \vert \hat{\nabla} \hat{\psi^{i}} \vert_{2} \int_{\hat{K}} \vert \hat{\chi_{1}} \vert_{2} \vert\hat{\chi_{2}}\vert d\hat{x} \notag
\\ &\leq Ch_{l}^{d-1} \|\hat{\chi_{1}}\|_{L2(\hat{K})} \|\hat{\chi_{2}}\|_{L2(\hat{K})} \notag
\\ &\leq Ch_{l}^{d-1} \|\hat{\nabla}\hat{\chi_{1}}\|_{L2(\hat{K})} \|\hat{\nabla}\hat{\chi_{2}}\|_{L2(\hat{K})}. \notag
\end{align}
Consider $\|\hat{\nabla}\hat{\chi_{2}}\|_{L2(\hat{K})}$ and $\|\hat{\nabla}\hat{\chi_{1}}\|_{L2(\hat{K})}$.  Transforming back to the mesh element and using standard FEM estimates yields
\begin{align} \label{chi2estimate}
\|\hat{\nabla}\hat{\chi_{2}}\|^{2}_{L2(\hat{K})} &= |det(B^{-1}_{K_{l}})| \int_{K_{l}} B^{T}_{K_{l}} \nabla \chi_{2} \cdot B^{T}_{K_{l}} \nabla \chi_{2} dx
\\ &\leq |det(B^{-1}_{K_{l}})| \|B^{T}_{K_{l}}\|^{2}_{op} \| \nabla \chi_{2} \|^{2}_{L2(K_{l})} \notag
\\ &\leq Ch_{l}^{2-d} \| \nabla \chi_{2} \|^{2}_{L2(K_{l})} \notag
\\ &\leq Ch_{l}^{2-d} \| \nabla \chi_{2} \|^{2}, \notag
\\ \|\hat{\nabla}\hat{\chi_{1}}\|^{2}_{L2(\hat{K})} &\leq Ch_{l}^{2-d} \| \nabla \chi_{1} \|^{2}. \label{chi1estimate}
\end{align}
Use (\ref{chi2estimate}) and (\ref{chi1estimate}) in (\ref{mainestimate}) and Young's inequality.  This yields
\begin{align*}
\frac{1}{2} \vert (\chi_{1}\cdot \nabla \tilde{T}^{i}\psi^{i},\chi_{2}) \vert \leq Ch_{l} \Big( \epsilon \|\nabla \chi_{1}\|^{2} + \epsilon^{-1} \|\nabla \chi_{2}\|^{2} \Big).
\end{align*}
Summing from $i=1$ to $i=i'$ and taking the maximum $h_{l}$ yields the result.
\end{proof}
\textbf{Remark:} If we allow the interpolant to be constructed with the basis elements of $W_{h}$, we can reconstruct any function $\upsilon_{h} \in W_{h}$ exactly on the boundary $\Gamma_{1}$ with the same properties.\\
\textbf{Remark:}  For square and cubic domains we can define such an interpolant explicitly, e.g.,
	\[ \tau(x) = 
	\begin{cases} 
	\frac{1}{2\delta}(2\delta- x_{\alpha}) & 0 \leq x_{\alpha} \leq \delta, \\
	\frac{1}{2} & \delta \leq x_{\alpha} \leq 1-\delta, \\
	\frac{1}{2\delta}(1- x_{\alpha}) & 1-\delta \leq x_{\alpha} \leq 1,
	\end{cases}
	\]
where $\alpha$ is in the direction orthogonal to the differentially heated walls or in the direction of gravity for the differentially heated vertical wall problem and Rayleigh-B\'{e}nard problem, respectively.  This function was introduced first by Hopf \cite{Hopf} and has been useful in estimating energy dissipation rates for shear-driven flows and convection \cite{Doering,Doering2}.
\begin{figure}
	\centering
	\includegraphics[height=1.5in, keepaspectratio]{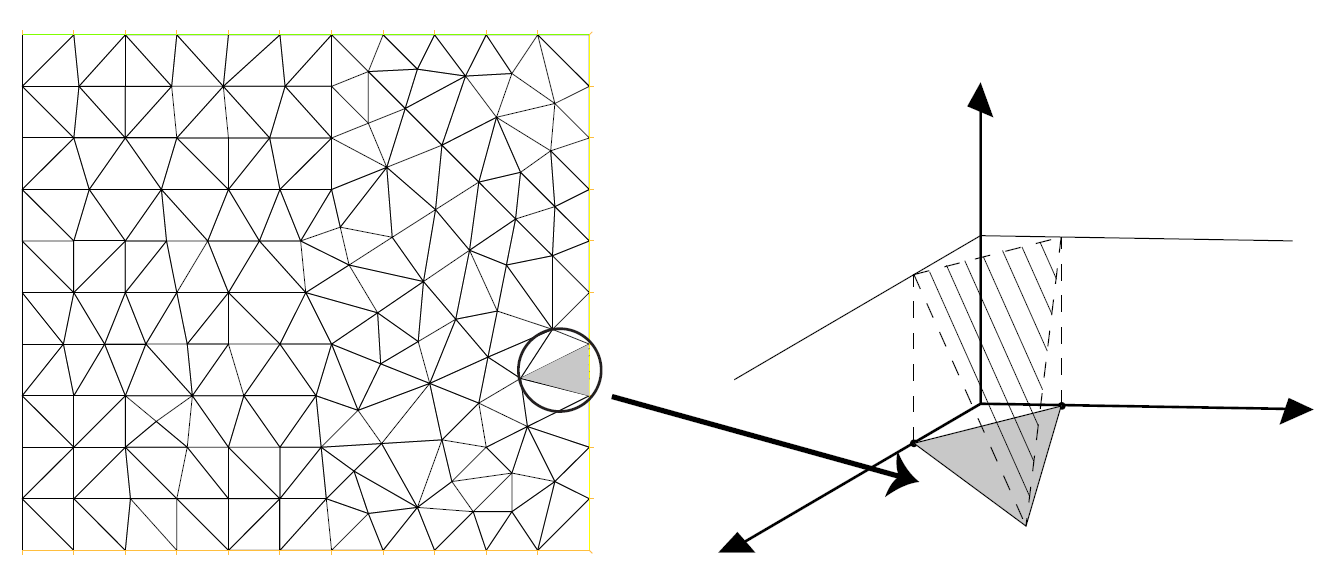}
	\caption{The discrete Hopf interpolant on one mesh element.}
\end{figure}
\section{Numerical Schemes}
In this section, we consider the following popular temporal discretizations: BDF1, linearly implicit BDF1, BDF2, and linearly implicit BDF2; see \cite{Ascher,Ingram} regarding linearly implicit variants.  Let $\eta(\chi) = a_{-1}\chi^{n+1} + a_{0}\chi^{n}$.  Denote the fully discrete solutions by $u^{n}_{h}$, $p^{n}_{h}$, and $T^{n}_{h}$ at time levels $t^{n} = n\Delta t$, $n = 0,1,...,N$, and $t^{\ast}=N\Delta t$.  Given $(u^{n}_{h}, T^{n}_{h})$ $\in (X_{h},W_{h})$, find $(u^{n+1}_{h}, p^{n+1}_{h}, T^{n+1}_{h})$ $\in (X_{h},Q_{h},W_{h})$ satisfying, for every $n = 0,1,...,N-1$, the fully discrete approximation of the system (\ref{ndeq1u}) - (\ref{ndeq1bc}) is

\noindent \textbf{BDF1 and linearly implicit BDF1:}
\begin{multline}\label{scheme:one:velocity}
(\frac{u^{n+1}_{h} - u^{n}_{h}}{\Delta t},v_{h}) + b(\eta(u_{h}),u^{n+1}_{h},v_{h}) + Pr(\nabla u^{n+1}_{h},\nabla v_{h})
\\ - (p^{n+1}_{h}, \nabla \cdot v_{h}) =  PrRa(\xi \eta(T_{h}),v_{h}) + (f^{n+1},v_{h}) \; \; \forall v_{h} \in X_{h},
\end{multline}
\begin{equation}
(\nabla \cdot u^{n+1}_{h},q_{h}) = 0 \; \; \forall q_{h} \in Q_{h},
\end{equation}
\begin{align}\label{scheme:one:temperature}
(\frac{T^{n+1}_{h} - T^{n}_{h}}{\Delta t},S_{h}) + b^{\ast}(\eta(u_{h}),T^{n+1}_{h},S_{h}) + (\nabla T^{n+1}_{h},\nabla S_{h}) = (\gamma^{n+1},S_{h}) \; \; \forall S_{h} \in W_{h,\Gamma_{1}},
\end{align}
where BDF1 is given by $a_{-1} = a_{0} +1 = 1$ and linearly implicit BDF1 by $a_{-1} + 1 = a_{0} = 1$.  Moreover, given $(u^{n-1}_{h}, T^{n-1}_{h})$ and $(u^{n}_{h}, T^{n}_{h})$ $\in (X_{h},Q_{h},W_{h})$, find $(u^{n+1}_{h}, p^{n+1}_{h}, T^{n+1}_{h})$ $\in (X_{h},Q_{h},W_{h})$ satisfying, for every $n = 1,2,...,N-1$, the fully discrete approximation of the system (\ref{ndeq1u}) - (\ref{ndeq1bc}) is
\\ \noindent \textbf{BDF2 and linearly implicit BDF2:}
\begin{multline}\label{scheme:three:velocity}
	(\frac{3u^{n+1}_{h} - 4u^{n}_{h} + u^{n-1}_{h}}{2\Delta t},v_{h}) + b(\eta(u_{h}),u^{n+1}_{h},v_{h}) + Pr(\nabla u^{n+1}_{h},\nabla v_{h}) - (p^{n+1}_{h}, \nabla \cdot v_{h})
	\\ =  PrRa(\xi \eta(T_{h}),v_{h}) + (f^{n+1},v_{h}) \; \; \forall v_{h} \in X_{h},
\end{multline}
\begin{equation}
(\nabla \cdot u^{n+1}_{h},q_{h}) = 0 \; \; \forall q_{h} \in Q_{h},
\end{equation}
\begin{align}\label{scheme:three:temperature}
(\frac{3T^{n+1}_{h} - 4T^{n}_{h} + T^{n-1}_{h}}{2\Delta t},S_{h}) + b^{\ast}(\eta(u_{h}),T^{n+1}_{h},S_{h}) + (\nabla T^{n+1}_{h},\nabla S_{h}) = (\gamma^{n+1},S_{h}) \; \; \forall S_{h} \in W_{h,\Gamma_{1}},
\end{align}
where BDF2 is given by $a_{-1} = a_{0} +1 = 1$ and linearly implicit BDF2 by $1 - a_{-1} = a_{0} = -1$. 
\section{Numerical Analysis}
We present stability results for the aforementioned algorithms provided the first meshline in the finite element mesh is within $\mathcal{O} (Ra^{-1})$ of the heated wall.

\subsection{Stability Analysis}

\begin{theorem} \label{t1}
Consider \textbf{BDF1} or \textbf{linearly implicit BDF1}.   Suppose $f \in L^{2}(0,t^{\ast};H^{-1}(\Omega)^{d})$, and $\gamma \in L^{2}(0,t^{\ast};H^{-1}(\Omega))$.  If $\delta = \mathcal{O} (Ra^{-1})$, then there exist $C > 0$, independent of $t^{\ast}$, such that
\\ \textbf{BDF:}
\begin{multline*}
\frac{1}{2}\|T^{N}_{h}\|^{2} + \|u^{N}_{h}\|^{2} + \sum_{n = 0}^{N-1}\|T^{n+1}_{h} - T^{n}_{h}\|^{2} + \sum_{n = 0}^{N-1}\|u^{n+1}_{h} - u^{n}_{h}\|^{2} + \frac{\Delta t}{4} \sum_{n = 0}^{N-1} \|\nabla T^{n+1}_{h}\|^{2}
\\ + \frac{Pr\Delta t}{4}\sum_{n = 0}^{N-1}\| \nabla u^{n+1}_{h} \|^{2} \leq Ct^{\ast},
\end{multline*}
\textbf{linearly implicit BDF:}
\begin{multline*}
\frac{1}{2}\|T^{N}_{h}\|^{2} + \|u^{N}_{h}\|^{2} + \sum_{n = 0}^{N-1}\|T^{n+1}_{h} - T^{n}_{h}\|^{2} + \sum_{n = 0}^{N-1}\|u^{n+1}_{h} - u^{n}_{h}\|^{2} + \frac{\Delta t}{4} \sum_{n = 0}^{N-1} \|\nabla T^{n+1}_{h}\|^{2}
\\ + \frac{Pr\Delta t}{8} \sum_{n = 0}^{N-1} \| \nabla u^{n+1}_{h} \|+ \frac{Pr\Delta t}{8}\| \nabla u^{N}_{h} \|^{2} \leq Ct^{\ast}.
\end{multline*}
\noindent Further,
\begin{equation*}
\beta \Delta t \sum^{N-1}_{n=0} \| p^{n+1}_{h}\| \leq C\sqrt{t^{\ast}}.
\end{equation*}
\end{theorem}
\begin{proof}
Our strategy is to first estimate the temperature approximation in terms of the velocity approximation and data.  We then bound the velocity approximation in terms of data yielding stability of both approximations.  Denote $\theta^{n+1}_{h} = T^{n+1}_{h} - \tau$.  Consider \textbf{BDF1}.  Let $S_{h} = \theta^{n+1}_{h} \in W_{\Gamma_{1},h}$ in equation (\ref{scheme:one:temperature}) and use the polarization identity.  Multiply by $\Delta t$ on both sides, rewrite all quantities in terms of $\theta^{k}_{h}$, $k = n,\;n+1$, and rearrange. Since $(\nabla \tau , \nabla\theta^{n+1}_{h}) = 0$ we have,
\begin{align} \label{stability:one}
\frac{1}{2} \Big\{\|\theta^{n+1}_{h}\|^{2} - \|\theta^{n}_{h}\|^{2} + \|\theta^{n+1}_{h} - \theta^{n}_{h}\|^{2}\Big\} + \Delta t\| \nabla \theta^{n+1}_{h} \|^{2} = - \Delta t b^{\ast}(u^{n+1}_{h}, \theta^{n+1}_{h} + \tau, \theta^{n+1}_{h})
\\ + \Delta t (\gamma^{n+1},\theta^{n+1}_{h}).\notag
\end{align}
Consider $-\Delta t b^{\ast}(u^{n+1}_{h}, \theta^{n+1}_{h} + \tau, \theta^{n+1}_{h})$.  Use skew-symmetry and apply  Lemma \ref{l1},
\begin{align}
-\Delta t b^{\ast}(u^{n+1}_{h}, \theta^{n+1}_{h} + \tau, \theta^{n+1}_{h}) = -\Delta t b^{\ast}(u^{n+1}_{h},\tau, \theta^{n+1}_{h}) \leq {C \Delta t \,\delta\, } \,\big( \epsilon^{-1}_{1} \| \nabla u^{n+1}_{h} \|^{2} + \epsilon_{1} \| \nabla \theta^{n+1}_{h} \|^{2}\big). \label{stability:one:est:bgflow}
\end{align}
Use Cauchy-Schwarz-Young on $\Delta t (\gamma^{n+1},\theta^{n+1}_{h})$,
\begin{align}
\Delta t (\gamma^{n+1},\theta^{n+1}_{h}) &\leq \frac{\Delta t}{2\epsilon_{2}} \|\gamma^{n+1}\|^{2}_{-1} + \frac{\Delta t \epsilon_{2}}{2} \|\nabla \theta^{n+1}_{h}\|^{2}. \label{stability:one:estg}
\end{align}
\noindent Using (\ref{stability:one:est:bgflow}) and (\ref{stability:one:estg}) in (\ref{stability:one}) leads to
\begin{multline*}
\frac{1}{2} \Big\{\|\theta^{n+1}_{h}\|^{2} - \|\theta^{n}_{h}\|^{2} + \|\theta^{n+1}_{h} - \theta^{n}_{h}\|^{2}\Big\} + \Delta t \|\nabla \theta^{n+1}_{h}\|^{2} \leq C\Delta t \,\delta\, \big( \epsilon^{-1}_{1} \| \nabla u^{n+1}_{h} \|^{2} + \epsilon_{1} \| \nabla \theta^{n+1}_{h} \|^{2}\big)
\\ + \frac{\Delta t}{2\epsilon_{2}} \|\gamma^{n+1}\|^{2}_{-1} + \frac{\Delta t \epsilon_{2}}{2} \|\nabla \theta^{n+1}_{h}\|^{2}.
\end{multline*}
\noindent Let $\epsilon_{1} = \frac{1}{2C\delta}$ and $\epsilon_{2} = 1/2$.   Regrouping terms leads to
\begin{align*}
\frac{1}{2} \Big\{\|\theta^{n+1}_{h}\|^{2} - \|\theta^{n}_{h}\|^{2} + \|\theta^{n+1}_{h} - \theta^{n}_{h}\|^{2}\Big\} +  \frac{\Delta t}{4} \|\nabla \theta^{n+1}_{h}\|^{2} \leq 2 C^{2}\Delta t \,\delta^{2}\, \| \nabla u^{n+1}_{h} \|^{2} + \Delta t \|\gamma^{n+1}\|^{2}_{-1}.
\end{align*}
\noindent Sum from $n = 0$ to $n = N-1$ and put all data on the right hand side.  This yields bounds on the temperature approximation in terms of the velocity approximation and data as follows, 
\begin{align} \label{keyeq}
\frac{1}{2} \|\theta^{N}_{h}\|^{2} + \frac{1}{2} \sum^{N-1}_{n=0} \|\theta^{n+1}_{h} - \theta^{n}_{h}\|^{2} +  \frac{\Delta t}{4} \sum^{N-1}_{n=0} \|\nabla \theta^{n+1}_{h}\|^{2} \leq 2 C^{2}\Delta t \,\delta^{2} \,\sum^{N-1}_{n=0} \| \nabla u^{n+1}_{h} \|^{2}
\\ + \Delta t \sum^{N-1}_{n=0} \|\gamma^{n+1}\|^{2}_{-1} + \frac{1}{2} \|\theta^{0}_{h}\|^{2}.\notag
\end{align}

Next, let $v_{h} = u^{n+1}_{h} \in V_{h}$ in (\ref{scheme:one:velocity}) and use the polarization identity.  Multiply by $\Delta t$ on both sides and rearrange terms.  Then,
\begin{align} \label{stability:oneu}
\frac{1}{2} \Big\{\|u^{n+1}_{h}\|^{2} - \|u^{n}_{h}\|^{2} + \|u^{n+1}_{h} - u^{n}_{h}\|^{2}\Big\} + Pr \Delta t  \|\nabla u^{n+1}_{h}\|^{2} = \Delta t PrRa(\xi (\theta^{n+1}_{h} + \tau), u^{n+1}_{h})
\\ + \Delta t (f^{n+1},u^{n+1}_{h}).\notag
\end{align}
Use the Cauchy-Schwarz-Young and Poincare-Friedrichs inequalities on $\Delta t PrRa(\xi (\theta^{n+1}_{h} + \tau), u^{n+1}_{h})$ and \\ $\Delta t (f^{n+1},u^{n+1}_{h})$ and note that $\| \xi \|_{L^{\infty}} = 1$,
\begin{align}
\Delta t PrRa(\xi \theta^{n+1}_{h}, u^{n+1}_{h}) &\leq \frac{\Delta t Pr^{2}Ra^{2}C_{PF,1}^2 C_{PF,2}^{2}}{2 \epsilon_{3}} \| \nabla \theta^{n+1}_{h} \|^{2} + \frac{\Delta t \epsilon_{3}}{2} \| \nabla u^{n+1}_{h} \|^{2}, \label{stability:one:estT}
\\ \Delta t PrRa(\xi \tau, u^{n+1}_{h}) &\leq \frac{\Delta t}{2\epsilon_{4}}Pr^2 Ra^2\|\tau\|^{2}_{-1} + \frac{\Delta t \epsilon_{4}}{2} \| \nabla u^{n+1}_{h} \|^{2}, \label{stability:one:esttau}
\\ \Delta t (f^{n+1},u^{n+1}_{h}) &\leq \frac{\Delta t}{2 \epsilon_{5}} \|f^{n+1} \|^{2}_{-1} + \frac{\Delta t \epsilon_{5}}{2} \| \nabla u^{n+1}_{h} \|^{2}. \label{stability:one:estf}
\end{align}
Using (\ref{stability:one:estT}), (\ref{stability:one:esttau}), and (\ref{stability:one:estf}) in (\ref{stability:oneu}) leads to
\begin{multline*}
\frac{1}{2} \Big\{\|u^{n+1}_{h}\|^{2} - \|u^{n}_{h}\|^{2} + \|u^{n+1}_{h} - u^{n}_{h}\|^{2}\Big\} + Pr \Delta t \|\nabla u^{n+1}_{h}\|^{2} + \leq \frac{\Delta t Pr^{2}Ra^{2}C_{PF,1}^2C_{PF,2}^2}{2 \epsilon_{3}} \| \nabla \theta^{n+1}_{h} \|^{2}
\\ + \frac{\Delta t Pr^2 Ra^2}{2\epsilon_{4}}\|\tau\|^{2}_{-1} + \frac{\Delta t}{2 \epsilon_{5}} \|f^{n+1} \|^{2}_{-1} + (\epsilon_{3} + \epsilon_{4} + \epsilon_{5})\,  \frac{\Delta t}{2}\|\nabla u^{n+1}_{h} \|^{2}.
\end{multline*}
Let $\epsilon_{3} = \epsilon_{4} = 4 \epsilon_{5} = Pr/2$.  Then,
\begin{multline*}
\frac{1}{2} \Big\{\|u^{n+1}_{h}\|^{2} - \|u^{n}_{h}\|^{2} + \|u^{n+1}_{h} - u^{n}_{h}\|^{2}\Big\} + \frac{Pr\Delta t}{4}\| \nabla u^{n+1}_{h} \| \leq \Delta t Pr Ra^{2}C_{PF,1}^2C_{PF,2}^2 \| \nabla \theta^{n+1}_{h} \|^{2}
\\ + {\Delta t Pr Ra^2}\|\tau\|^{2}_{-1} + \frac{\Delta t}{Pr} \|f^{n+1} \|^{2}_{-1}.
\end{multline*}
\noindent Summing from $n = 0$ to $n = N-1$ and putting all data on r.h.s. yields
\begin{multline} \label{stability:vel}
\frac{1}{2}\|u^{N}_{h}\|^{2} +\frac{1}{2}\sum_{n = 0}^{N-1}\|u^{n+1}_{h} - u^{n}_{h}\|^{2} + \frac{Pr\Delta t}{4}\sum_{n = 0}^{N-1}\| \nabla u^{n+1}_{h} \|\leq \Delta t Pr Ra^{2}C_{PF,1}^2C_{PF,2}^2 \sum_{n = 0}^{N-1} \| \nabla \theta^{n+1}_{h} \|^{2}
\\ + \frac{\Delta t}{Pr} \sum_{n = 0}^{N-1} \Big( Pr^2 Ra^2\|\tau\|^{2}_{-1} + \|f^{n+1} \|^{2}_{-1} \Big)+\frac{1}{2}\|u^{0}_{h}\|^{2}.
\end{multline}

Now, from equation (\ref{keyeq}), we have
\begin{multline}\label{stability:coupling}
\Delta t Pr Ra^{2}C_{PF,1}^2C_{PF,2}^2 \sum_{n = 0}^{N-1} \| \nabla \theta^{n+1}_{h} \|^{2} \leq\,8 \,C^2  C_{PF,1}^2C_{PF,2}^2Pr Ra^{2} \,\delta^{2}\, \Delta t \sum^{N-1}_{n=0} \| \nabla u^{n+1}_{h} \|^{2}
\\ + 4 Pr Ra^{2}C_{PF,1}^2C_{PF,2}^2 \Delta t \sum^{N-1}_{n=0} \|\gamma^{n+1}\|^{2}_{-1} + 2 Pr Ra^{2}C_{PF,1}^2C_{PF,2}^2 \|\theta^{0}_{h}\|^{2}.
\end{multline}
Using the above in (\ref{stability:vel}) with $\delta = \frac{1}{8\, C\, C_{PF,1}C_{PF,2}} \, Ra^{-1}$ leads to
\begin{multline} \label{stability:one:vel}
\frac{1}{2}\|u^{N}_{h}\|^{2} +\frac{1}{2}\sum_{n = 0}^{N-1}\|u^{n+1}_{h} - u^{n}_{h}\|^{2} + \frac{Pr\Delta t}{8}\sum_{n = 0}^{N-1}\| \nabla u^{n+1}_{h} \| 
\\ \leq 4 Pr Ra^{2}C_{PF,1}^2C_{PF,2}^2 \Delta t \sum^{N-1}_{n=0} \|\gamma^{n+1}\|^{2}_{-1} + 2 Pr Ra^{2}C_{PF,1}^2C_{PF,2}^2 \|\theta^{0}_{h}\|^{2}
\\ + \frac{\Delta t}{Pr} \sum_{n = 0}^{N-1} \Big( Pr^2 Ra^2 \|\tau\|^{2}_{-1} + \|f^{n+1} \|^{2}_{-1} \Big) + \frac{1}{2}\|u^{0}_{h}\|^{2}.
\end{multline}

Thus, the velocity approximation is bounded above by data and therefore the temperature approximation as well; that is, both the velocity and temperature approximations are stable.  Adding (\ref{keyeq}) and (\ref{stability:one:vel}), multiplying by 2, and using the identity $T^{n}_{h} = \theta^{n}_{h} + \tau$ together with the triangle inequality yields the result.

Next, consider \textbf{linearly implicit BDF1}.  We apply similar techniques as in the above. This leads to
\begin{align} \label{keyeq2}
\frac{1}{2} \|\theta^{N}_{h}\|^{2} + \frac{1}{2} \sum^{N-1}_{n=0} \|\theta^{n+1}_{h} - \theta^{n}_{h}\|^{2} +  \frac{\Delta t}{4} \sum^{N-1}_{n=0} \|\nabla \theta^{n+1}_{h}\|^{2} \leq 4C^{2}\Delta t \delta^{2} \sum^{N-1}_{n=0} \| \nabla u^{n}_{h} \|^{2}
\\ + \Delta t \sum^{N-1}_{n=0} \|\gamma^{n+1}\|^{2}_{-1} + \frac{1}{2} \|\theta^{0}_{h}\|^{2},\notag
\end{align}
\noindent and
\begin{multline} \label{stability:one:vel2}
\frac{1}{2}\|u^{N}_{h}\|^{2} +\frac{1}{2}\sum_{n = 0}^{N-1}\|u^{n+1}_{h} - u^{n}_{h}\|^{2} + \frac{Pr\Delta t}{8} \sum_{n = 0}^{N-1} \| \nabla u^{n+1}_{h} \| + \frac{Pr\Delta t}{8}\| \nabla u^{N}_{h} \| 
\\ \leq 4 Pr Ra^{2}C_{PF,1}^2C_{PF,2}^2 \Delta t \sum^{N-1}_{n=0} \|\gamma^{n+1}\|^{2}_{-1} + 2 Pr Ra^{2}C_{PF,1}^2C_{PF,2}^2 \|\theta^{0}_{h}\|^{2}
\\ + \frac{\Delta t}{Pr} \sum_{n = 0}^{N-1} \Big( \|\tau\|^{2}_{-1} + \|f^{n+1} \|^{2}_{-1} \Big) + \frac{1}{2}\|u^{0}_{h}\|^{2} + \frac{Pr\Delta t}{8}\| \nabla u^{0}_{h} \|.
\end{multline}
The result follows.  We now prove stability of the pressure approximation.  Consider (\ref{scheme:two:velocity}), isolate \\ $(\frac{u^{n+1}_{h} - u^{n}_{h}}{\Delta t},v_{h})$, let $0 \neq v_{h} \in V_{h}$, and multiply by $\Delta t$.  Then,
\begin{multline}\label{vv}
	({u^{n+1}_{h} - u^{n}_{h}},v_{h}) = -\Delta t b(\eta(u_{h}),u^{n+1}_{h},v_{h}) - \Delta t Pr(\nabla u^{n+1}_{h},\nabla v_{h}) + \Delta t PrRa(\xi \eta(T_{h}),v_{h}) + \Delta t (f^{n+1},v_{h}).
\end{multline}
Applying Lemma \ref{l2} to the skew-symmetric trilinear term and the Cauchy-Schwarz and Poincar\'{e}-Friedrichs inequalities to the remaining terms yields
\begin{align}
|-\Delta t b(\eta(u_{h}),u^{n+1}_{h},v_{h})| &\leq C_{1} \Delta t \| \nabla \eta(u_{h}) \| \| \nabla u^{n+1}_{h} \| \| \nabla v_{h}\|,\label{stability:one:estmean}\\
|-\Delta t Pr(\nabla u^{n+1}_{h},\nabla v_{h})| &\leq Pr \Delta t \|\nabla u^{n+1}_{h} \| \|\nabla v_{h} \|, \label{stability:one:estvisc}\\
|\Delta t PrRa(\xi \eta(T_{h}),v_{h})| &\leq PrRaC_{PF,1} \Delta t \|\eta(T_{h}) \| \|\nabla v_{h} \|, \label{stability:one:estcoupling}\\
|\Delta t(f^{n+1},v_{h})| &\leq \Delta t \|f^{n+1} \|_{-1} \|\nabla v_{h} \| \label{stability:one:estf2}.
\end{align}
Apply the above estimates in (\ref{vv}), divide by the common factor $\|\nabla v_{h} \|$ on both sides, and take the supremum over all $0 \neq v_{h} \in V_{h}$.  Then,
\begin{align}
\|u^{n+1}_{h} - u^{n}_{h}\|_{V^{\ast}_{h}} \leq C_{1} \Delta t \| \nabla \eta(u_{h}) \| \| \nabla u^{n+1}_{h} \| + Pr \Delta t \|\nabla u^{n+1}_{h} \| +  PrRaC_{PF,1} \Delta t \|\eta(T_{h}) \| + \Delta t \|f^{n+1} \|_{-1}.
\end{align}
\noindent Reconsider equations (\ref{scheme:one:velocity}) and (\ref{scheme:two:velocity}).  Multiply by $\Delta t$ and isolate the pressure term,
\begin{multline}
\Delta t (p^{n+1}_{h}, \nabla \cdot v_{h}) = (u^{n+1}_{h} - u^{n}_{h},v_{h}) + \Delta t b(\eta(u_{h}),u^{n+1}_{h},v_{h}) + Pr \Delta t(\nabla u^{n+1}_{h},\nabla v_{h})
\\ - PrRa \Delta t(\xi \eta(T_{h}),v_{h})  - \Delta t(f^{n+1},v_{h}).
\end{multline}
\noindent Apply (\ref{stability:one:estmean}), (\ref{stability:one:estvisc}), (\ref{stability:one:estcoupling}), and (\ref{stability:one:estf2}) on the r.h.s terms.  Then,
\begin{multline}
\Delta t (p^{n+1}_{h}, \nabla \cdot v_{h}) \leq (u^{n+1}_{h} - u^{n}_{h},v_{h}) + \Big(C_{1} \Delta t \| \nabla \eta(u_{h}) \| \| \nabla u^{n+1}_{h} \| + Pr \Delta t \|\nabla u^{n+1}_{h} \|
\\ +  PrRaC_{PF,1} \Delta t \|\eta(T_{h}) \| + \Delta t \|f^{n+1} \|_{-1}\Big)\|\nabla v_{h} \|.
\end{multline}
\noindent Divide by $\|\nabla v_{h} \|$ and note that $\frac{(u^{n+1}_{h} - u^{n}_{h},v_{h})}{\| \nabla v_{h} \|} \leq \|u^{n+1}_{h} - u^{n}_{h}\|_{V^{\ast}_{h}}$.  Take the supremum over all $0 \neq v_{h} \in X_{h}$,
\begin{multline}
\Delta t \sup_{0 \neq v_{h} \in X_{h}}\frac{(p^{n+1}_{h}, \nabla \cdot v_{h})}{\|\nabla v_{h} \|} \leq 2 \Big(C_{1} \Delta t \| \nabla \eta(u_{h}) \| \| \nabla u^{n+1}_{h} \| + Pr \Delta t \|\nabla u^{n+1}_{h} \|
\\ +  PrRaC_{PF,1} \Delta t \|\eta(T_{h}) \| + \Delta t \|f^{n+1} \|_{-1}\Big).
\end{multline}
\noindent Use the inf-sup condition (\ref{infsup}),
\begin{multline}
\beta \Delta t \| p^{n+1}_{h}\| \leq 2 \Big( C_{1} \Delta t \| \nabla \eta(u_{h}) \| \| \nabla u^{n+1}_{h} \| + Pr \Delta t \|\nabla u^{n+1}_{h} \|
\\ +  PrRaC_{PF,1} \Delta t \|\eta(T_{h}) \| + \Delta t \|f^{n+1} \|_{-1}\Big).
\end{multline}
\noindent Summing from $n = 0$ to $n = N-1$ yields stability of the pressure approximation, built on the stability of the temperature and velocity approximations.
\end{proof}

\begin{theorem} \label{t2}
Consider \textbf{BDF2} or \textbf{linearly implicit BDF2}.   Suppose $f \in L^{2}(0,\infty;H^{-1}(\Omega)^{d})$, and $\gamma \in L^{2}(0,\infty;H^{-1}(\Omega))$.  If $\delta = \mathcal{O} (Ra^{-1})$, then there exist $C > 0$, independent of $t^{\ast}$, such that
\\ \textbf{BDF2:}
\begin{multline*}
\frac{1}{2} \|T^{N}_{h}\|^{2} + \frac{1}{2} \|2T^{N}_{h} - T^{N-1}_{h}\|^{2} + \|u^{N}_{h}\|^{2} + \|2u^{N}_{h} - u^{N-1}_{h}\|^{2} + \sum_{n = 1}^{N-1} \|T^{n+1}_{h} - 2T^{n}_{h} + T^{n-1}_{h}\|^{2} 
\\ + \sum_{n = 1}^{N-1} \|u^{n+1}_{h} - 2u^{n}_{h} + u^{n-1}_{h}\|^{2} + \frac{\Delta t}{2} \sum_{n = 1}^{N-1} \| \nabla T^{n+1}_{h} \|^{2} + \frac{Pr \Delta t}{2} \sum_{n = 1}^{N-1} \| \nabla u^{n+1}_{h} \|^{2} \leq Ct^{\ast},
\end{multline*}
\textbf{linearly implicit BDF2:}
\begin{multline*}
\frac{1}{2} \|T^{N}_{h}\|^{2} + \frac{1}{2} \|2T^{N}_{h} - T^{N-1}_{h}\|^{2} + \|u^{N}_{h}\|^{2} + \|2u^{N}_{h} - u^{N-1}_{h}\|^{2} + \sum_{n = 1}^{N-1} \|T^{n+1}_{h} - 2T^{n}_{h} + T^{n-1}_{h}\|^{2} 
\\ + \sum_{n = 1}^{N-1} \|u^{n+1}_{h} - 2u^{n}_{h} + u^{n-1}_{h}\|^{2} + \frac{\Delta t}{2} \sum_{n = 1}^{N-1} \| \nabla T^{n+1}_{h} \|^{2} + \frac{Pr \Delta t}{2} \sum_{n = 1}^{N-1} \| \nabla u^{n+1}_{h} \|^{2}
\\ + \frac{Pr \Delta t}{2} \Big(\| \nabla u^{N}_{h} \|^{2} + \| \nabla u^{N-1}_{h} \|^{2}\Big) \leq Ct^{\ast}.
\end{multline*}
\noindent Further,
\begin{equation*}
\beta \Delta t \sum^{N-1}_{n=0} \| p^{n+1}_{h}\| \leq C\sqrt{t^{\ast}}.
\end{equation*}
\end{theorem}
\begin{proof}
We follow the general strategy in Theorem \ref{t1}.  Consider \textbf{linearly implicit BDF2} first.  Let $S_{h} = \theta^{n+1}_{h} \in W_{\Gamma_{1},h}$ in equation (\ref{scheme:four:temperature}) and use the polarization identity.  Multiply by $\Delta t$ on both sides, rewrite all quantities in terms of $\theta^{k}_{h}$, $k = n,\;n+1$, and rearrange. Then,
\begin{multline}\label{stability:four}
\frac{1}{4} \Big\{\|\theta^{n+1}_{h}\|^{2} + \|2\theta^{n+1}_{h} - \theta^{n}_{h}\|^{2}\Big\} - \frac{1}{4} \Big\{\|\theta^{n}_{h}\|^{2} + \|2\theta^{n}_{h} - \theta^{n-1}_{h}\|^{2}\Big\} + \frac{1}{4} \|\theta^{n+1}_{h} - 2\theta^{n}_{h} + \theta^{n-1}_{h}\|^{2}
\\ +  \Delta t \|\nabla \theta^{n+1}_{h}\|^{2} = - \Delta t b^{\ast}(2u^{n}_{h}-u^{n-1}_{h},\tau,\theta^{n+1}_{h}) + \Delta t (\gamma^{n+1},\theta^{n+1}_{h}).
\end{multline}
Consider $-\Delta t b^{\ast}(2u^{n}_{h}-u^{n-1}_{h},\tau,\theta^{n+1}_{h}) = -2\Delta t b^{\ast}(u^{n}_{h},\tau,\theta^{n+1}_{h}) + \Delta t b^{\ast}(u^{n-1}_{h},\tau,\theta^{n+1}_{h})$.  Use Lemma \ref{t2}, then
\begin{align}
-2\Delta t b^{\ast}(u^{n}_{h},\tau,\theta^{n+1}_{h}) &\leq C \,\delta\, \Delta t
\Big( 4\epsilon^{-1}_{6} \| \nabla u^{n}_{h} \|^{2} +  \epsilon_{6} \| \nabla \theta^{n+1}_{h} \|^{2} \Big),
\\ \Delta t b^{\ast}(u^{n-1}_{h},\tau,\theta^{n+1}_{h}) &\leq C \,\delta\, \Delta t
\Big( \epsilon^{-1}_{7} \| \nabla u^{n-1}_{h} \|^{2} +  \epsilon_{7} \| \nabla \theta^{n+1}_{h} \|^{2} \Big).
\end{align}
Use above estimates and (\ref{stability:one:estg}) in equation (\ref{stability:four}).  Let $\epsilon_{6} = \epsilon_{7} = \frac{1}{4C\delta}$ and $\epsilon_{2} = 1/4$.  This leads to
\begin{multline}
\frac{1}{4} \Big\{\|\theta^{n+1}_{h}\|^{2} + \|2\theta^{n+1}_{h} - \theta^{n}_{h}\|^{2}\Big\} - \frac{1}{4} \Big\{\|\theta^{n}_{h}\|^{2} + \|2\theta^{n}_{h} - \theta^{n-1}_{h}\|^{2}\Big\} + \frac{1}{4} \|\theta^{n+1}_{h} - 2\theta^{n}_{h} + \theta^{n-1}_{h}\|^{2}
\\ + \frac{\Delta t}{4} \|\nabla \theta^{n+1}_{h}\|^{2} \leq 16C^{2}\Delta t \,\delta^{2}\,\| \nabla u^{n}_{h} \|^{2} + 4C^{2}\Delta t \,\delta^{2}\, \| \nabla u^{n-1}_{h} \|^{2} + 2\Delta t \|\gamma^{n+1}\|^{2}_{-1}.
\end{multline}
\noindent Sum from $n = 1$ to $n = N-1$ and put all data on the right hand side.  This yields
\begin{multline} \label{keyeq4}
\frac{1}{4} \|\theta^{N}_{h}\|^{2} + \frac{1}{4} \|2\theta^{N}_{h} - \theta^{N-1}_{h}\|^{2} + \frac{1}{4} \sum^{N-1}_{n=1}\|\theta^{n+1}_{h} - 2\theta^{n}_{h} + \theta^{n-1}_{h}\|^{2} + \frac{\Delta t}{4} \sum^{N-1}_{n=1} \|\nabla \theta^{n+1}_{h}\|^{2}
\\ \leq 16C^{2}\Delta t \,\delta^{2}\, \sum^{N-1}_{n=1}\| \nabla u^{n}_{h} \|^{2} + 4C^{2}\Delta t \,\delta^{2}\, \sum^{N-1}_{n=1}\| \nabla u^{n-1}_{h} \|^{2} + 2\Delta t \sum^{N-1}_{n=1}\|\gamma^{n+1}\|^{2}_{-1}
\\ + \frac{1}{4} \|\theta^{0}_{h}\|^{2} + \frac{1}{4} \|\theta^{1}_{h} - \theta^{0}_{h}\|^{2}.
\end{multline}

Now, let $v_{h} = u^{n+1}_{h} \in V_{h}$ in (\ref{scheme:four:velocity}) and use the polarization identity.  Multiply by $\Delta t$ on both sides and rearrange terms.  Then,
\begin{multline} \label{stability:oneu4}
\frac{1}{4} \Big\{\|u^{n+1}_{h}\|^{2} + \|2u^{n+1}_{h} - u^{n}_{h}\|^{2}\Big\} - \frac{1}{4} \Big\{\|u^{n}_{h}\|^{2} + \|2u^{n}_{h} - u^{n-1}_{h}\|^{2}\Big\} + \frac{1}{4} \|u^{n+1}_{h} - 2u^{n}_{h} + u^{n-1}_{h}\|^{2}
\\ + Pr \Delta t \|\nabla u^{n+1}_{h}\|^{2} = \Delta t PrRa(\xi (2\theta^{n}_{h} - \theta^{n-1}_{h} + \tau), u^{n+1}_{h}) + \Delta t (f^{n+1},u^{n+1}_{h}).
\end{multline}
Use the Cauchy-Schwarz-Young and Poincare-Friedrichs inequalities on $\Delta t PrRa(\xi (2\theta^{n}_{h} - \theta^{n-1}_{h} + \tau), u^{n+1}_{h})$,
\begin{align}
2 \Delta t PrRa(\xi \theta^{n}_{h}, u^{n+1}_{h}) &\leq \frac{4\Delta t Pr^{2}Ra^{2}C_{PF,1}^2 C_{PF,2}^{2}}{2 \epsilon_{8}} \| \nabla \theta^{n}_{h} \|^{2} + \frac{\Delta t \epsilon_{8}}{2} \| \nabla u^{n+1}_{h} \|^{2}, \label{stability:one:estT4}
\\ -\Delta t PrRa(\xi \theta^{n-1}_{h}, u^{n+1}_{h}) &\leq \frac{\Delta t Pr^{2}Ra^{2}C_{PF,1}^2 C_{PF,2}^{2}}{2\epsilon_{9}} \| \nabla \theta^{n-1}_{h} \|^{2} + \frac{\Delta t \epsilon_{9}}{2} \| \nabla u^{n+1}_{h} \|^{2}. \label{stability:one:estT5}
\end{align}
Using (\ref{stability:one:esttau}), (\ref{stability:one:estf}), (\ref{stability:one:estT4}), and (\ref{stability:one:estT5}) in (\ref{stability:oneu4}) leads to
\begin{multline*}
\frac{1}{4} \Big\{\|u^{n+1}_{h}\|^{2} + \|2u^{n+1}_{h} - u^{n}_{h}\|^{2}\Big\} - \frac{1}{4} \Big\{\|u^{n}_{h}\|^{2} + \|2u^{n}_{h} - u^{n-1}_{h}\|^{2}\Big\} + \frac{1}{4} \|u^{n+1}_{h} - 2u^{n}_{h} + u^{n-1}_{h}\|^{2}
\\ + Pr \Delta t \|\nabla u^{n+1}_{h}\|^{2} \leq \frac{2\Delta t Pr^{2}Ra^{2}C_{PF,1}^2C_{PF,2}^2}{\epsilon_{8}} \| \nabla \theta^{n}_{h} \|^{2} + \frac{\Delta t Pr^{2}Ra^{2}C_{PF,1}^2C_{PF,2}^2}{2\epsilon_{9}} \| \nabla \theta^{n-1}_{h} \|^{2}
\\ + \frac{\Delta t}{2\epsilon_{4}}\|\tau\|^{2}_{-1} + \frac{\Delta t}{2 \epsilon_{5}} \|f^{n+1} \|^{2}_{-1} + \frac{\Delta t}{2}(\epsilon_{4} + \epsilon_{5} + \epsilon_{8} + \epsilon_{9})\|\nabla u^{n+1}_{h} \|^{2}.
\end{multline*}
Let $2\epsilon_{4} = 2 \epsilon_{5} = \epsilon_{8} = \epsilon_{9} = Pr/2$.  Then,
\begin{multline*}
\frac{1}{4} \Big\{\|u^{n+1}_{h}\|^{2} + \|2u^{n+1}_{h} - u^{n}_{h}\|^{2}\Big\} - \frac{1}{4} \Big\{\|u^{n}_{h}\|^{2} + \|2u^{n}_{h} - u^{n-1}_{h}\|^{2}\Big\} + \frac{1}{4} \|u^{n+1}_{h} - 2u^{n}_{h} + u^{n-1}_{h}\|^{2}
\\ + \frac{Pr \Delta t}{4} \|\nabla u^{n+1}_{h}\|^{2} \leq 4\Delta t Pr Ra^{2}C_{PF,1}^2C_{PF,2}^2 \| \nabla \theta^{n}_{h} \|^{2} + \Delta t Pr^{2}Ra^{2}C_{PF,1}^2C_{PF,2}^2 \| \nabla \theta^{n-1}_{h} \|^{2}
\\ + \frac{2\Delta t}{Pr}\|\tau\|^{2}_{-1} + \frac{2\Delta t}{Pr} \|f^{n+1} \|^{2}_{-1}.
\end{multline*}
\noindent Summing from $n = 1$ to $n = N-1$ and putting all data on r.h.s. yields
\begin{multline} \label{stability:vel4}
\frac{1}{4} \|u^{N}_{h}\|^{2} + \frac{1}{4} \|2u^{N}_{h} - u^{N-1}_{h}\|^{2} + \frac{1}{4} \sum_{n = 1}^{N-1} \|u^{n+1}_{h} - 2u^{n}_{h} + u^{n-1}_{h}\|^{2}
\\ + \frac{Pr \Delta t}{4} \sum_{n = 1}^{N-1} \|\nabla u^{n+1}_{h}\|^{2} \leq \Delta t Pr Ra^{2}C_{PF,1}^2C_{PF,2}^2 \sum_{n = 1}^{N-1} \Big( 4 \| \nabla \theta^{n}_{h} \|^{2} + \| \nabla \theta^{n-1}_{h} \|^{2} \Big)
\\ + \frac{2\Delta t}{Pr} \sum_{n = 1}^{N-1} \Big( \|\tau\|^{2}_{-1} + \|f^{n+1} \|^{2}_{-1} \Big) + \frac{1}{4} \|u^{1}_{h}\|^{2} + \frac{1}{4} \|2u^{1}_{h} - u^{0}_{h}\|^{2}.
\end{multline}
\noindent Now, from equation (\ref{keyeq4}), we have
\begin{multline}\label{stability:coupling4}
\Delta t Pr Ra^{2}C_{PF,1}^2C_{PF,2}^2 \sum_{n = 1}^{N-1} \| \nabla \theta^{n+1}_{h} \|^{2} \leq 64\,C^2 \, C_{PF,1}^2C_{PF,2}^2 Pr Ra^{2} \,\delta^{2}\, \Delta t \sum^{N-1}_{n=0} \Big(\| \nabla u^{n}_{h} \|^{2} + \| \nabla u^{n-1}_{h} \|^{2}\Big)
\\ + 8 Pr Ra^{2}C_{PF,1}^2C_{PF,2}^2 \Delta t \sum^{N-1}_{n=1} \|\gamma^{n+1}\|^{2}_{-1} + Pr Ra^{2}C_{PF,1}^2C_{PF,2}^2 \Big(\|\theta^{1}_{h}\|^{2} + \|2\theta^{1}_{h}-\theta^{0}_{h}\|^{2} \Big).
\end{multline}
Add and subtract $\frac{Pr\Delta t}{8}\sum_{n = 1}^{N-1}\| \nabla u^{n}_{h} \|$ and $\frac{Pr\Delta t}{8}\sum_{n = 1}^{N-1}\| \nabla u^{n-1}_{h} \| $ in (\ref{stability:vel4}) and use the above estimate with $\delta = \frac{1}{16\sqrt{2}\,C\, C_{PF,1}C_{PF,2}} \, Ra^{-1}$.  Then,
\begin{multline} \label{stability:one:vel4}
\frac{1}{4} \|u^{N}_{h}\|^{2} + \frac{1}{4} \|2u^{N}_{h} - u^{N-1}_{h}\|^{2} + \frac{1}{4} \sum_{n = 1}^{N-1} \|u^{n+1}_{h} - 2u^{n}_{h} + u^{n-1}_{h}\|^{2} + \frac{Pr\Delta t}{8}\sum_{n = 1}^{N-1}\| \nabla u^{n+1}_{h} \| + \frac{Pr\Delta t}{8}\| \nabla u^{N}_{h} \|
\\ + \frac{Pr\Delta t}{8}\| \nabla u^{N-1}_{h} \| \leq 8 Pr Ra^{2}C_{PF,1}^2C_{PF,2}^2 \Delta t \sum^{N-1}_{n=1} \|\gamma^{n+1}\|^{2}_{-1} + Pr Ra^{2}C_{PF,1}^2C_{PF,2}^2 \Big(\|\theta^{1}_{h}\|^{2} + \|2\theta^{1}_{h}-\theta^{0}_{h}\|^{2} \Big)
\\ + \frac{2\Delta t}{Pr} \sum_{n = 0}^{N-1} \Big( \|\tau\|^{2}_{-1} + \|f^{n+1} \|^{2}_{-1} \Big)  + \frac{1}{4} \|u^{1}_{h}\|^{2} + \frac{1}{4} \|2u^{1}_{h} - u^{0}_{h}\|^{2} + \frac{Pr\Delta t}{8}\| \nabla u^{1}_{h} \| + \frac{Pr\Delta t}{8}\| \nabla u^{0}_{h} \| 
\end{multline}

The result follows.  Applying similar techniques as in the above and Theorem \ref{t1} yields the result for \textbf{BDF2}.  Pressure stability follows by similar arguments in Theorem \ref{t1}.
\end{proof}

\section{Conclusion}

The coupling terms $b^{\ast}(\eta(u_{h}),T^{n+1}_{h},S_{h})$ and $PrRa(\xi \eta(T),v_{h})$ that arise in stability analyses of FEM discretizations of natural convection problems with sidewall heating are the major source of difficulty.  The former term forces the stability of the temperature approximation to be dependent on the velocity approximation and vice versa for the latter term.  Standard techniques fail to overcome this imposition, in the absence of a discrete Gronwall inequality. \\
\indent The authors introduced a new discrete Hopf interpolant that was able to overcome this issue.  Fully discrete stability estimates were proven which improve upon previous estimates.  In particular, it was shown that provided that the first mesh line in the finite element mesh is within $\mathcal{O} (Ra^{-1})$ of the nonhomogeneous Dirichlet boundary, the velocity, pressure and temperature approximations are stable allowing for sub-linear growth in $t^{\ast}$.\\  
\indent A uniform in time stability estimate was not able to be achieved due to the term $PrRa(\xi \tau,v_{h})$, which arises when an interpolant of the boundary is introduced.  The authors conjecture that the results proven herein may be improved, owing to a gap in the analysis.  \textbf{Open problems include:} Is it possible to improve the current results with a less restrictive mesh condition?  Moreover, can these results be improved to uniform in time stability?  An important next step would be reanalyzing stability for natural convection problems, with sidewall heating, where a turbulence model is incorporated.  
\section*{Acknowledgements}
The authors would like to thank Professor William Layton for suggesting this problem and for many fruitful discussions. Both authors were partially supported by NSF grants DMS 1522267 and CBET 1609120.  Moreover, J.A.F. is supported by the DoD SMART Scholarship.

\end{document}